\documentclass[a4paper,12pt]{article}

\usepackage{amssymb}

\usepackage{amsthm}
\usepackage{amsmath}
\theoremstyle{plain}
\newtheorem{theorem}{Theorem}[section]

\newtheorem{corollary}[theorem]{Corollary}

\newtheorem{remark}[theorem]{Remark}

 \title{On expansions involving the Riemann zeta function and its derivatives}
\author{Lahoucine Elaissaoui\footnote{Department of mathematics, Faculty of sciences, Mohammed V University in Rabat, 4 Street Ibn Battouta B.P. 1014 RP, Rabat, e-mail: \texttt{l.elaissaoui@um5r.ac.ma}}}

\begin{document}




\maketitle



\begin{center}
{\small \textit{To the memory of my father Brahim, 1956--2021, because of him I love mathematics. }}
\end{center}

\vspace*{1cm}
\begin{abstract}
By studying the spectral aspects of the fractional part function in a well-known separable Hilbert space, we show, among other things, a rational approximation of the Riemann zeta function and its derivatives valid on every vertical line in the right half-planes $\Re s> 1/2$ and $\Re s >0.$ Moreover, we provide some discussions and explicit computations related to the fractional part function. 
\end{abstract}



\section{Introduction and statement of the main result}
\label{}
The distribution of values of the Riemann zeta function $\zeta(s)$ in the complex plane $\mathbb{C}$ is an important topic in number theory in view of its strong dependence on the behavior of some interesting classes of arithmetic functions. In fact, the value-distribution of $\zeta(s)$ inside and on the boundary of the critical strip, $0 < \Re s < 1,$ is far from being understood completely. We recall that, the Riemann zeta function is initially defined in the half-plane $\Re s > 1$ by $ \zeta(s) = \sum_{n \geq 1} 1/n^s$ and has an analytic continuation to the whole complex plane except for a simple pole at $s=1$ of residue $1,$ thanks to the functional equation \cite[p. 16]{Titch}   
\begin{equation}
\zeta(s) = \chi(s) \zeta(1-s) ,
\label{FEZ}
\end{equation}
where
$$\chi(s) = \pi^{s-\frac{1}{2}} \frac{\Gamma\left( \frac{1-s}{2}\right)}{\Gamma\left( \frac{s}{2}\right)} $$
and $\Gamma(s)$ is the well-known gamma function. As you notice, the functional equation \eqref{FEZ} implies a connection between the values of $\zeta(s)$ and of $\zeta(1-s)$ with the so-called `` critical line", $\Re s = 1/2,$ as vertical and symmetry axis. Of course, the value-distribution of the function $\chi(s)$ on every vertical line $\Re s = \sigma,$ for any bounded $\sigma,$ is well understood thanks to Stirling's formula; namely, for $s = \sigma + it$ where $\sigma$ is any bounded real number and $|t|\geq 2,$ we have 
$$\chi(s) = \left| \frac{t}{2\pi}\right|^{\frac{1}{2}-\sigma } \exp\left( i \left( \frac{\pi}{4} - t \log \frac{t}{2\pi e}\right)\right) \left( 1 + O\left( \frac{1}{|t|} \right)\right) .$$
Thus, one can deduce by using the Phragmén-Lindel\"of principle \cite[\S 9.41]{Titch2} that if $\zeta\left( \frac{1}{2} + it\right) = O \left( |t|^{\lambda + \varepsilon} \right)$  for any $\varepsilon >0$ and for a sufficiently large real $|t|,$ then we have uniformly in the strip $ 1/2 \leq \sigma < 1,$
$$\zeta(s) = O \left( |t|^{2 \lambda (1-\sigma) + \varepsilon} \right) \qquad \forall \varepsilon >0 .$$
Therefore, the growth rate of $\zeta(s)$ in the rest of the critical strip follows from the functional equation \eqref{FEZ}. The best estimate of $\lambda$ is due, recently, to Bourgain \cite{Bour} who obtained $\lambda = 13/84.$ However, the famous Lindel\"of hypothesis affirms that the best estimate is for $\lambda =0;$ i.e. for a sufficiently large $|t|,$
$$\zeta\left( \frac{1}{2} + i t \right) = O\left( |t|^{\varepsilon}\right), \qquad \forall \varepsilon >0. $$ 
      
Recently, the author and Guennoun showed in \cite[Th. 3.1]{Ela} the following rational approximation of $\zeta(s);$ for any given complex number $s$ in the half-plane $\sigma > 1/2,$ 
\begin{equation}
\zeta(s) = \frac{s}{s-1} + \sum_{n = 0}^{+\infty} \ell_n \left( \frac{1-s}{s} \right), 
\label{Elah}
\end{equation}
where $(\ell_n)$ are the Fourier coefficients of the Riemann zeta function that are shown to be, explicitly, the binomial transform of the Stieltjes constants $(\gamma_n);$ i.e.
\begin{equation}
    \ell_n = \sum_{k=1}^n \binom{n-1}{k-1} \frac{(-1)^{n-k}}{k!}\gamma_k, \qquad \forall n \in \mathbb{N}
\label{lnBT}\end{equation}
and $\ell_0 = \gamma_0 - 1$ where $\gamma_0\approx 0.577$ is the well-known Euler-Mascheroni constant; we denoted by $\mathbb{N}$ the set of positive integers and we put $\mathbb{N}_0 =\mathbb{N} \cup \{0\}  .$  We recall that the Stieltjes constants $(\gamma_n)$ are defined by
$$\gamma_n = \lim_{N \to + \infty} \left( \sum_{k=1}^N \frac{\log^n k}{k} - \frac{\log^{n+1}N}{n+1}\right);$$ 
including the case when $n=0$ to define the Euler-Mascheroni constant. Actually, the explicit from of the Fourier coefficients $(\ell_n)$ follows from its integral representation (see \cite[eq. (9)]{Ela})
\begin{equation}
\ell_n = \frac{1}{2\pi i} \int_{\Re s = \frac{1}{2}} \frac{\zeta(s)}{s(1-s)}\left( \frac{s}{1-s}\right)^n \mathrm{d}s.
\label{BigEq}
\end{equation}
Furthermore, the author and Guennoun showed that the Fourier coefficients $(\ell_n)$ belongs to the classical Hilbert space of the square-summable sequences, denoted as usually by $\ell^2(\mathbb{N});$ and 
$$\|(\ell_n)\|_2^2 = \sum_{n=0}^{+ \infty} \ell_n^2 = \log(2\pi) - \gamma_0 - 1. $$ 
However, one can prove by using the ``falsity" of Lindel\"of's boundedness conjecture (see for example \cite[Cor. 10.2, p. 184]{Edw}) that the sequences $(\ell_n)$ converges conditionally, namely
$$\sum_{n=0}^{+\infty} \ell_n = \zeta\left( \frac{1}{2}\right)+ 1 \qquad \mbox{and}\qquad \sum_{n=0}^{+\infty} |\ell_n|= + \infty .$$  
We should not forget to mention that, the equality \eqref{Elah} holds true on the critical line $\sigma = \frac{1}{2};$ but, in this case, the convergence of the series in the right-hand side of \eqref{Elah} is uniform on any compact in the real line $\mathbb{R}$ (see \cite[Rmk. 2.1]{Ela}). 
 
Finally, we notice that the behavior of the Riemann zeta function on every vertical line $ \sigma> 1/2$ is strongly related to the growth rate of the sequence $(\ell_n).$ In this paper, we show a similar expansion for the derivatives of the Riemann zeta function on the half-plane $\sigma > 1/2;$, which could be considered as a generalization of \eqref{Elah};
\begin{theorem}
 For any given complex number $s$ in the half-plane $\sigma > 1/2$ and any $m \in \mathbb{N}_0,$ we have uniformly
 $$\sum_{k=0}^m\frac{(-s)^k}{k!}\zeta^{(k)}(s) = \left(\frac{s}{s-1}\right)^{m+1} + \sum_{n=0}^{+\infty}\binom{n+m}{m}\ell_n^{(m)} \left(\frac{1-s}{s} \right)^n, $$
 where $\zeta^{(k)}$ denotes the $k-$th derivative of the Riemann zeta function; in particular $\zeta^{(0)}(s)=\zeta(s),$  and
 $$\ell_n^{(m)}:=  - \delta_{n,0} + \sum_{k=0}^n \binom{n}{k}(-1)^{n-k}\sum_{j=0}^{m+k}\frac{\gamma_j}{j!}, \qquad n \in \mathbb{N}_0;$$
  $\delta_{n,k}$ denotes the classical Kronecker delta symbol given by $\delta_{n,k} = 1$ or $0$ according, respectively, to $k = n$ or $k \neq n.$ 
\label{ThEL}\end{theorem}

  One can deduce directly, from Theorem \ref{ThEL}, a rational approximation of the $m-$th derivative of the Riemann zeta function that is valid for any complex number $s$ in the half-palne $\Re s > 1/2;$ as we are showing in the next section. Namely,
\begin{corollary}
Let $m \in \mathbb{N}.$ Then for any complex number $s,$ $\Re s >1/2,$ we have
$$\zeta^{(m)}(s) = \frac{(-1)^m m!}{(s-1)^{m+1}} + \frac{(-1)^m m!}{s^{m+1}} \sum_{n=1}^{+\infty}\binom{n+m-1}{m} \ell_n^{(m-1)} \left( \frac{s-1}{s}\right)^{n-1}.  $$
\label{Coroprin}\end{corollary}  
For the case $m=0,$ we have the following rational approximation of the Riemann zeta function, which can be considered as an other generalization of the main result obtained in \cite{Ela}.
\begin{theorem}
For any complex number $s,$ in the right-half plane $\Re s > 0$ we have
$$\zeta(s) = \frac{s}{s-1} - \frac{1}{2} - \frac{s}{s+\frac{1}{2}} \sum_{n=0}^{+\infty} \alpha_n \left( \frac{s-\frac{1}{2}}{s+\frac{1}{2}}\right)^n; $$
where
$$ \alpha_n  \underset{n \to +\infty}{=}o(1), $$
are given in proof's section.
\label{THELA}\end{theorem}
  
   Moreover, it follows that the distribution of values of the Riemann zeta function and its derivatives are encoded in the Stieltjes constants $(\gamma_n).$ However, the general behavior of $(\gamma_n)$ is not yet understood completely. We shall show, in the next section, that the Stieltjes constants (which are strictly related to the coefficients $(\ell_n^{(k)})$) depend in reality on the coordinates of the fractional part function in some Hilbert spaces; thus the proof of Theorem \ref{ThEL} will be deduced. The last section consists of explicit values of integrals involving fractional part function with a brief discussion.

 \section{Fourier-Laguerre expansion of fractional part function and the proof of the main theorems}
 
 As the author and Guennoun showed in \cite{Ela}, the behavior of the sequences $(\ell_n)$  is of great interest in the theory of the Riemann zeta function since relevant information on the distribution of values of $\zeta(s)$ on every vertical line located in the right half-plane of the critical line.  In this section, we establish the strong connection of $(\ell_n)$ with the well-known Laguerre polynomials which allows us to provide a new series representation of the fractional part function denoted, for any given real variable $x,$ by $\{x\} = x-\lfloor x\rfloor \in [0,1)$ where $\lfloor \cdot \rfloor$ denotes the classical integer part function (or floor function).

 \subsection{The series representation of the fractional part function}

 We begin with the integral representation of the Fourier coefficient $(\ell_n).$
 
 \begin{theorem}
 For any integer $n \geq 0,$ we have
 $$ \ell_n =(-1)^{n-1} \int_1^{+\infty}\frac{\{x\}}{x^2}\mathcal{L}_n( x)\mathrm{d}x ,$$
 where $\mathcal{L}_n(x)=L_n(\log x)$ and $L_n$ are the well-known Laguerre polynomials (see for example \cite[Ch. 23]{OMS}).
 \label{Th21}\end{theorem}
 \begin{proof}
 Let us first start with an integral representation of the Stieltjes constants $(\gamma_n)_{n \geq 1}.$ Thus, for a given integer $N \geq 1$ we have 
 \begin{align*}\sum_{k=1}^N\frac{\log^n k}{k} = \int_{1^{-}}^N\frac{\log^n x}{x} \mathrm{d}\lfloor x \rfloor &= \log^n N - \int_1^N \frac{ n\log^{n-1}x - \log^nx }{x^2}\lfloor x\rfloor \mathrm{d}x \\ &= \frac{\log^{n+1}N}{n+1} + \int_1^N \frac{n\log^{n-1}x - \log^nx }{x^2}\{ x\} \mathrm{d}x .\end{align*}
 Hence, for any $n \in \mathbb{N}$
$$
     \gamma_n = \lim_{N \to + \infty} \left( \sum_{k=1}^N \frac{\log^n k}{k} - \frac{\log^{n+1}N}{n+1}\right) =  \int_1^{+\infty} \frac{n\log^{n-1}x - \log^nx }{x^2}\{ x\} \mathrm{d}x .
$$

 Now, for any integer $n \geq 1,$ we have, \eqref{lnBT},
 \begin{align*}
     (-1)^n\ell_n &= \sum_{k=1}^n \binom{n-1}{k-1} \frac{(-1)^{k}}{k!}\gamma_k \\ &= \int_1^{+\infty} \left(\sum_{k=1}^n \binom{n-1}{k-1} \frac{(-1)^{k}}{k!}\left(k\log^{k-1}x - \log^kx \right)\right)\frac{\{ x\}  }{x^2}\mathrm{d}x .
 \end{align*}
 Since, for all $n \in \mathbb{N} $ and any fixed real $x > 1,$
 \begin{align*}
     \sum_{k=1}^n &\binom{n-1}{k-1} \frac{(-1)^{k}}{k!}\left(k\log^{k-1}x - \log^kx \right)\\ &= -1+\frac{(-1)^{n-1}}{n!}\log^n x + \sum_{k=1}^{n-1}\left(\binom{n-1}{k}+\binom{n-1}{k-1}\right) \frac{(-1)^{k-1}}{k!}\log^kx  \\ &= -1+\frac{(-1)^{n-1}}{n!}\log^nx + \sum_{k=1}^{n-1} \binom{n}{k}\frac{(-1)^{k-1}}{k!}\log^k x\\ &= \sum_{k=0}^{n} \binom{n}{k}\frac{(-1)^{k-1}}{k!}\log^kx = - \mathcal{L}_n(x);
 \end{align*}
 Then, we obtain an integral representation of $(\ell_n),$ namely

\begin{equation}(-1)^{n-1}\ell_n = \int_1^{+\infty}\frac{\{x\}}{x^2}\mathcal{L}_n( x)\mathrm{d}x; \qquad n\in \mathbb{N}_0. 
\label{IRl}\end{equation}
Notice that the case of $n=0$ is included,
$$\ell_0 = - \int_1^{+\infty} \frac{\{x\}}{x^2}\mathrm{d}x = \gamma_0-1. $$
\end{proof}

Actually, Laguerre polynomials are of great interest in quantum mechanics because it represent the radial part of the wave function, the solution of the Schr\"odinger equation for the hydrogen-like atom (see for more details \cite{MR,RS}). Maybe this fact could explain the presence of the Riemann zeta function in many quantum mechanics and theoretical physics problems.
For the sake of completeness, let us briefly recall some properties on the `` modified Laguerre basis" $\{\mathcal{L}_n(x)\}_{n \geq 0}$ and its associated functions $\left\{\mathcal{L}_n^{(m)}(x)\right\}_{n \geq 0}$ ($m \in \mathbb{N}_0$), which are defined, for any $n\in \mathbb{N}_0$ and $x\in(1,+\infty),$ by
\begin{equation}
\mathcal{L}_n^{(m+1)}(x) := \sum_{k=0}^n \mathcal{L}_k^{(m)}(x) \qquad \mbox{and}\qquad \mathcal{L}_n^{(0)}(x) := \mathcal{L}_n(x) .
\label{DfLm}   
\end{equation}
Remark that $\mathcal{L}_n^{(m)}(x)=L_n^{(m)}(\log x),$ where $L_n^{(m)}$ are the associated Laguerre polynomials. For the case when $m >-1$ is real, the polynomials $L_n^{(m)}$ are called the generalized Laguerre polynomials \cite[Ch 5.]{Sz}. Thus, the explicit form of the functions $\mathcal{L}_n^{(m)}$ is given by, \cite[eq. 5.1.6]{Sz}, 
\begin{equation}
\mathcal{L}_n^{(m)}(x) = \sum_{k=0}^{n}\binom{n+m}{n-k} \frac{(-1)^k}{k!} \log^kx    
\label{F1}\end{equation}
 for all $m,n \in \mathbb{N}_0,$ and satisfying the following three-term recurrence relation \cite[eq. 5.1.10]{Sz}, 
$$
(n+1)\mathcal{L}_{n+1}^{(m)}(x)=(2n+1+m-\log x)\mathcal{L}_n^{(m)}(x) - (n+m) \mathcal{L}_{n-1}^{(m)}(x),      
$$
 with $\mathcal{L}_0^{(m)}(x)=1.$ Moreover, we would like to point out that by using \cite[eq. 5.1.14]{Sz} and \eqref{DfLm} one can show that
$$\frac{\mathrm{d}}{\mathrm{d}x}\mathcal{L}_n^{(m)}(x) = - \frac{\mathcal{L}_{n-1}^{(m+1)}(x)}{x} \qquad\mbox{and then}\qquad \frac{\mathrm{d}}{\mathrm{d}x} \left( \frac{\mathcal{L}_{n}^{(m)}(x)}{x}\right) = -\frac{\mathcal{L}_{n}^{(m+1)}(x)}{x^2}.$$
Finally, the generating function corresponding to associated Laguerre polynomials $(L_n^{(m)})_{n\geq 0}$ is given by \cite[eq. 5.1.9]{Sz}
$$\sum_{n=0}^{+\infty} L_n^{(m)}(u) z^n = \frac{1}{(1-z)^{m+1}}\exp\left(\frac{uz}{z-1} \right), $$
for any complex $z$ such that $|z|<1$ and $u\in (0,+\infty).$ Therefore, by putting $u=\log x$  for any $x \in (1,+\infty)$ and $z = (s-1)/s$ for $\Re s > 1/2,$  we obtain 

\begin{equation}\sum_{n=0}^{+\infty} \mathcal{L}_n^{(m)}(x) \left( \frac{s-1}{s} \right)^n = \frac{s^{m+1}}{x^{s-1}}; \label{GFm}
\end{equation}
where the convergence is absolute if and only if $\Re s > 1/2,$ for any given real $x\geq 1.$ In particular, for $m=0,$ we have
\begin{equation} \sum_{n=0}^{+\infty} \mathcal{L}_n(x) \left( \frac{s-1}{s} \right)^n = \frac{s}{x^{s-1}},\qquad  \Re s > \frac12 .\label{GF0}\end{equation}
Thus, by using Theorem \ref{ThEL} one can provide an other proof of \eqref{Elah}. In fact, for all $\Re s > 1/2,$
\begin{align*}
    \sum_{n=0}^{+\infty} \ell_n \left( \frac{1-s}{s}\right)^n &= -\sum_{n=0}^{+\infty} \int_1^{+\infty} \frac{\{x\}}{x^2}\mathcal{L}_n(x)\mathrm{d}x\left(\frac{s-1}{s} \right)^n \\ &=- \int_1^{+\infty} \frac{\{x\}}{x^2}\left(\sum_{n=0}^{+\infty} \mathcal{L}_n(x)\left(\frac{s-1}{s} \right)^n\right)\mathrm{d}x \\ &= -s\int_1^{+\infty} \frac{\{x\}}{x^{s+1}}\mathrm{d}x\\ &=\zeta(s) - \frac{s}{s-1}.
\end{align*}
Notice that, the last equality is well-known and can be found for example in \cite{Titch}. We point out, that by applying the Mellin inversion theorem (see for example \cite{Titch3}) to the formula
\begin{align}
    \zeta(s)  &= \frac{s}{s-1} - s\int_1^{+\infty} \frac{\{x\}}{x^{s+1}}\mathrm{d}x \label{ZFM}\\ &= -s \int_0^{+\infty} \frac{\{x\}}{x^{s+1}}\mathrm{d}x, \qquad \Re s \in (0,1) \nonumber
\end{align}
we obtain, for any given $\sigma_0 \in (0,1) $ and all real $x>0$
\begin{equation}-\frac{1}{2\pi i}\int_{\Re s = \sigma_0}\zeta(s)\frac{x^s}{s} \mathrm{d}s = \begin{cases}\begin{array}{cl} \frac{1}{2} & \mbox{if} \ x \in \mathbb{N} \\ \{x\} & \mbox{otherwise}. \end{array} \end{cases} \label{IMF}\end{equation}
Remark that, the last integral does not converge in the Lebesgue sense; however, it is understood as the Cauchy principal value. 

Next, it is well-known that the associated Laguerre polynomials $\left\{L_n^{(m)}(x)\right\}_{n\geq 0}$ ($m \in \mathbb{N}_0$) form a complete orthogonal set in the Hilbert space $\mathrm{L}^2\left((0,+\infty); \omega_m\right);$ where $\omega_m(x)=x^{m}e^{-x}/m!,$ \cite[Th. 4.2]{AK}. Thus, since the function $x\mapsto \log x$ bijectively maps the interval $(1, +\infty)$ to the interval $(0,+\infty)$ then the functions $\left\{\mathcal{L}_n^{(m)}(x)\right\}_{n \geq 0}$ define an orthogonal basis in the Hilbert space
\begin{align*}\mathcal{H}_m &:= \mathrm{L}^2\left((1,+\infty); \frac{\log^m x}{m! x^2}\right)\\ &= \left\{ f: (1,+\infty) \rightarrow \mathbb{R},\quad \|f\|_{(m)}^2:= \frac{1}{m!} \int_1^{+\infty} \frac{|f(x)|^2\log^m x}{x^2}\mathrm{d}x < +\infty \right\}. \end{align*}
The associated inner product is defined, for any $f,g \in \mathcal{H}_m,$ by 
$$\langle f , g\rangle_{(m)} := \frac{1}{m!}\int_1^{+\infty} \frac{f(x)g(x)\log^mx}{x^2}\mathrm{d}x. $$
Moreover, one can show, by using  integration by substitution in \cite[eq. 5.1.1]{Sz}, that for any fixed $m \in \mathbb{N}_0$ and for all $k,n \in \mathbb{N}_0$
$$ \langle \mathcal{L}_n^{(m)} , \mathcal{L}_k^{(m)}\rangle_{(m)} = \binom{n+m}{m}\delta_{n,k}. $$
Therefore, any function $f \in \mathcal{H}_m$ can be expanded as
\begin{equation}
f(x) = \sum_{n=0}^{+\infty} c_{n,f}^{(m)}\mathcal{L}_n^{(m)}(x) \qquad\mbox{where} \qquad c_{n,f}^{(m)}=\frac{\langle f, \mathcal{L}_n^{(m)} \rangle_{(m)}}{\binom{n+m}{m}};     
\label{F3}\end{equation}
and for all $f,g \in\mathcal{H}_m$

$$ \langle f , g\rangle_{(m)} = \sum_{n=0}^{+\infty} \binom{n+m}{m}c_{n,f}^{(m)}c_{n,g}^{(m)}<+\infty.$$
In particular, Parseval's identity asserts that,
$$\| f\|_{(m)}^2 = \sum_{n=0}^{+\infty}\binom{n+m}{m}\left| c_{n,f}^{(m)}\right|^2 < +\infty, \qquad \forall f \in \mathcal{H}_m .$$
 As an application, we obtain the Fourier-Laguerre expansion of the fractional part function.
\begin{theorem}
For any real $x>1,$ we have the following expansion
$$\sum_{n=0}^{+\infty} (-1)^{n-1}\ell_n\mathcal{L}_n(x) = \begin{cases}\begin{array}{cl} \frac{1}{2} & \mbox{if} \ x \in \mathbb{N}\setminus\{1\} \\ \{x\} & \mbox{otherwise}. \end{array} \end{cases}$$
where $(\ell_n)_{n \geq 0}$ are the Fourier coefficients of the Riemann zeta function.
\label{Th22}
\end{theorem}
\begin{proof}
It is clear that the fractional part function belongs to the Hilbert space $\mathcal{H}_0;$ and
$$\|\{\cdot\}\|_{(0)}^2 = \int_1^{+\infty} \frac{\{x\}^2}{x^2}\mathrm{d}x = \log(2\pi) - \gamma_0 - 1. $$
Thus by Theorem \ref{Th21}, we have
$$\langle \{\cdot\} , \mathcal{L}_n \rangle_{(0)} = \int_1^{+\infty}\frac{\{x\}}{x^2}\mathcal{L}_n(x)\mathrm{d}x = (-1)^{n-1}\ell_n .$$
Now, by moving the line of integration in \eqref{BigEq} to the line $\Re s = \sigma_0$ for any fixed $\sigma_0 \in (0,1/2)$ and since the function $\zeta(s)s^{n-1}/(1-s)^{n+1}$ is holomorphic in the strip $[\sigma_0,1/2),$ then for any integer $n \geq 0$ and real $x>1$
$$\sum_{k=0}^n (-1)^{k-1}\ell_k \mathcal{L}_k(x) = \frac{1}{2\pi i} \int_{\Re s = \sigma_0} \frac{\zeta(s)}{s(1-s)}\left( \sum_{k=0}^n \mathcal{L}_k(x) \left(\frac{s}{s-1} \right)^k\right)\mathrm{d}s.$$
Notice that $|s/(s-1)|<1$ for all $\Re s \in [\sigma_0,1/2)$ and that the function $\zeta(s)/(s(1-s))$ is absolutely integrable on the line $\Re s =\sigma_0.$ Then by the uniform convergence of the generating function \eqref{GF0}; for all $\Re s \in [\sigma_0, 1/2)$ and any real $x > 1,$
$$\lim_{n \to +\infty} \sum_{k=0}^n \mathcal{L}_k(x) \left(\frac{s}{s-1} \right)^k = (1-s)x^s, $$
we obtain
$$\sum_{n = 0}^{+\infty} (-1)^{n-1}\ell_n \mathcal{L}_n(x) = - \frac{1}{2\pi i} \int_{\Re s = \sigma_0} \frac{\zeta(s)}{s}x^s \mathrm{d}s. $$
The formula \eqref{IMF} completes the proof.
\end{proof}
\begin{remark}
Notice that, the series in Theorem \ref{Th22} does not converge uniformly. Indeed, the sequences of functions $\left((-1)^{n-1}\ell_n \mathcal{L}_n\right)_{n \geq 0}$ are continuous, then if the series in Theorem \ref{Th22} converges uniformly then the limit must be a continuous function which is a contradiction.
\end{remark}

Actually, one can generalize Theorem \ref{Th22} by showing that the fractional part function belongs to the Hilbert space $\mathcal{H}_m,$ for any given $m \in \mathbb{N}_0,$ and then the following expansion holds in $\mathcal{H}_m-$norm,
$$\{x\} =  \sum_{n=0}^{+\infty} (-1)^{n-1}\ell_{n}^{(m)}\mathcal{L}_n^{(m)}(x); $$
where
$$ (-1)^{n-1}\ell_{n}^{(m)} := \frac{1}{m! \binom{n+m}{m}}\int_1^{+\infty} \frac{\{x\} \log^m x \mathcal{L}_n^{(m)}(x)}{x^2}\mathrm{d}x;$$
with $\ell_n^{(0)} = \ell_n.$ \\
It should remark that, the series in the representation \eqref{F3} converges in the $\mathrm{L}^2-$norm sense; it need not, in general, hold in a pointwise sense. However, one can find many alternative results on the convergence theory of Laguerre series, for example, in \cite{Hil1,Hil2}. 
\subsection{Proof of Theorem \ref{ThEL}}

Let $f(x,s)=\{x\}/x^{s+1}$ for all $x>1$ and $\Re s > 1/2.$ Then the function $x \mapsto f(x,s)$ is Lebesgue integrable for any complex number $s$ such that $\Re s > 1/2$ and $s\mapsto f(x,s)$ is holomorphic on the half-plane $\Re s > 1/2.$ Moreover, for all $x>1$ and any given $m \in \mathbb{N}_0,$ we have
$$\left|\frac{\mathrm{d}^{m}}{\mathrm{d}s^m}f(x,s)\right| = \left| (-1)^m \log^mx \, f(x,s)\right| \leq  \frac{\{x\}\log^m x}{x^{\frac{3}{2}}}, \quad \Re s > \frac{1}{2}. $$
Since the function $x\mapsto \{x\}/x^{3/2}$ is Lebesgue integrable on $(1,+\infty)$, then we have, for all $\Re s > 1/2$
$$\frac{\mathrm{d}^{m}}{\mathrm{d}s^m} \int_1^{+\infty} \frac{\{x\}}{x^{s+1}}\mathrm{d}x = (-1)^m \int_1^{+\infty}\frac{\{x\}\log^m x}{x^{s+1}}\mathrm{d}x. $$
It follows by \eqref{ZFM} that 
$$\frac{\mathrm{d}^m}{\mathrm{d}s^m}\left( \frac{\zeta(s)}{s} - \frac{1}{s-1}\right) = (-1)^{m-1}\int_1^{+\infty}\frac{\{x\}\log^m x}{x^{s+1}}\mathrm{d}x. $$
Since, 
\begin{align*}
    \frac{\mathrm{d}^m}{\mathrm{d}s^m}\left( \frac{\zeta(s)}{s} - \frac{1}{s-1}\right) &= \sum_{k=0}^m \binom{m}{k}\frac{(-1)^{m-k}(m-k)!}{s^{m-k+1}}\zeta^{(k)}(s) - \frac{(-1)^m m!}{(s-1)^{m+1}} \\ &= \frac{(-1)^m m!}{s^{m+1}}\left( \sum_{k=0}^m\frac{(-s)^k}{k!}\zeta^{(k)}(s) - \left(\frac{s}{s-1}\right)^{m+1}\right),
\end{align*}
then for any given $m\in \mathbb{N}_0$ and any complex number $s$ in the half-plane $\Re s> 1/2$ we obtain
$$\sum_{k=0}^m\frac{(-s)^k}{k!}\zeta^{(k)}(s) - \left(\frac{s}{s-1}\right)^{m+1} = -\frac{s^{m+1}}{m!} \int_1^{+\infty}\frac{\{x\}\log^m x}{x^{s+1}}\mathrm{d}x. $$
Notice that, the case of $s=1$ is included as the limit as $s$ tend to $1;$ namely, the Laurent expansion of $\zeta(s)$ near its pole $s=1$ is given by
$$\zeta(s) - \frac{1}{s-1} = \sum_{k=0}^{+\infty}\frac{(-1)^k\gamma_k}{k!}(s-1)^k, $$
thus, we have near $s=1,$ for any given $m\in \mathbb{N}_0,$ 
\begin{align*} \frac{\mathrm{d}^m}{\mathrm{d}s^m}\left( \frac{\zeta(s)}{s} - \frac{1}{s-1}\right) &= \frac{\mathrm{d}^m}{\mathrm{d}s^m}\left( \frac{\zeta(s)}{s} - \frac{1}{s(s-1)} - \frac{1}{s}\right) \\ &= \frac{\mathrm{d}^m}{\mathrm{d}s^m}\left( \sum_{k=0}^{+\infty}\frac{(-1)^k\gamma_k}{k!}\frac{(s-1)^k}{s} \right) - \frac{(-1)^m m!}{s^{m+1}} \\ &\underset{s\to 1}{=} (-1)^m m!\left(\sum_{k=0}^m\frac{\gamma_k}{k!} - 1\right).
\end{align*}
Hence, for any given $m \in \mathbb{N}_0,$
\begin{equation}
 \lim_{s \to 1}  \left(\sum_{k=0}^m\frac{(-s)^k}{k!}\zeta^{(k)}(s) - \left(\frac{s}{s-1}\right)^{m+1}\right)  = -1 + \sum_{k=0}^m\frac{\gamma_k}{k!}= \ell_0^{(m)}.  
\label{L0M}\end{equation}  
Now, since the functions $\{\cdot \}$ and $x \mapsto s^{m+1}/x^{s+1}$ ($\Re s > 1/2$) belong to the Hilbert space $\mathcal{H}_m,$ for any given $m \in \mathbb{N}_0,$ then by \eqref{GFm} we have
$$\frac{s^{m+1}}{m!} \int_1^{+\infty}\frac{\{x\}\log^m x}{x^{s+1}}\mathrm{d}x = \sum_{n=0}^{+\infty} (-1)^{n-1}\binom{n+m}{m}\ell_n^{(m)}\left( \frac{s-1}{s}\right)^n .$$
It remains to show the explicit form of $\ell_n^{(m)}$ given in the statement of Theorem \ref{ThEL}. In fact, it follows by \eqref{F1} that
\begin{align*} \int_1^{+ \infty} \frac{\{x\} \log^m x \mathcal{L}_n^{(m)}(x)}{x^2} \mathrm{d}x &= \sum_{k=0}^n \binom{n+m}{k+m}\frac{(-1)^k}{k!} \int_1^{+\infty} \frac{\{x\}\log^{k+m}x}{x^2}\mathrm{d}x \\ &= \frac{(n+m)!}{n!} \sum_{k=0}^n\binom{n}{k}(-1)^{k-1} \ell_0^{(m+k)},\end{align*}
remark that, by \eqref{L0M}
$$\ell_0^{(m+k)} = \frac{-1}{(m+k)!} \int_1^{+\infty}\frac{\{x\}\log^{m+k}x}{x^2}\mathrm{d}x = -1 + \sum_{j=0}^{m+k}\frac{\gamma_j}{j!}. $$
Thus, 
 \begin{align*}\sum_{k=0}^n\binom{n}{k}(-1)^{k-1} \ell_0^{(m+k)}&= \sum_{k=0}^n\binom{n}{k}(-1)^{k} + \sum_{k=0}^n\binom{n}{k}(-1)^{k-1}\sum_{j=0}^{m+k}\frac{\gamma_{j}}{j!} \\ &= \delta_{n,0} + \sum_{k=0}^n\binom{n}{k}(-1)^{k-1}\sum_{j=0}^{m+k}\frac{\gamma_{j}}{j!}.
 \end{align*}
Finally, the fact that
$$\ell_n^{(m)} = \frac{(-1)^{n-1}}{m!\binom{n+m}{m}}\int_1^{+\infty}\frac{\{x\}\log^m x \mathcal{L}_n^{(m)}(x)}{x^2}\mathrm{d}x $$
for any given $m \in \mathbb{N}_0$ completes the proof of Theorem \ref{ThEL}.
\begin{remark}
The convergence of the series in Theorem \ref{ThEL} is absolute for any complex number $s$ in the half-plane $\Re s > 1/2.$ In fact, by using Cauchy-Schwarz inequality we have 
\begin{align*}
\sum_{n = 0}^{+\infty} \binom{n+m}{n}|\ell_n^{(m)}|\left| \frac{1-s}{s}\right|^n & \leq \left(\sum_{n=0}^{+\infty}\binom{n+m}{m}\left(\ell_n^{(m)}\right)^2\right)^{\frac{1}{2}}\left( \sum_{n = 0}^{+\infty} \binom{n+m}{n}\left| \frac{1-s}{s}\right|^{2n} \right)^{\frac{1}{2}} \\ &= \|\{\cdot\}\|_{(m)} \left(\frac{|s|}{\sqrt{2\sigma - 1}}\right)^{m+1}.
\end{align*}
\end{remark}
\subsection{Proof of Corollary \ref{Coroprin}.}
Let $m \in \mathbb{N}$ and $\Re s >1/2$ with $s \neq 1.$ It follows from Theorem \ref{ThEL} that,
$$\frac{(-s)^m}{m!}\zeta^{(m)}(s) = \left( \frac{s}{s-1}\right)^m\frac{1}{s-1} + \sum_{n=0}^{+\infty}\Delta_{n}^{(m)}\left( \frac{1-s}{s}\right)^n, $$
where
\begin{align*}
\Delta_n^{(m)} &= \binom{n+m}{m}\ell_n^{(m)}-\binom{n+m-1}{m-1}\ell_n^{(m-1)} \\&=\frac{(-1)^{n-1}}{m!}\int_1^{+\infty}\frac{\{x\}\log^{m-1}x}{x^2}\left(\log x \mathcal{L}_n^{(m)}(x) - m \mathcal{L}_n^{(m-1)}(x) \right) \mathrm{d}x.
\end{align*}
Since, for all $x \in [1,+\infty),$
$$-\left(\log x \mathcal{L}_n^{(m)}(x) - m \mathcal{L}_n^{(m-1)}(x)\right) = (n+1)\mathcal{L}_{n+1}^{(m-1)} - n \mathcal{L}_n^{(m-1)},  $$
then
$$\Delta_n^{(m)} = \frac{(-1)^n}{m!} \int_1^{+\infty}\frac{\{x\}\log^{m-1}x}{x^2}\left( (n+1)\mathcal{L}_{n+1}^{(m-1)} - n \mathcal{L}_n^{(m-1)}\right)\mathrm{d}x.$$
Now, for any given large $N \in \mathbb{N},$ we have
\begin{align*} \sum_{n=0}^N\Delta_n^{(m)}\left( \frac{1-s}{s}\right)^n &= \sum_{n=0}^N\left(d_{n+1}^{(m)}-d_n^{(m)}\right)\left( \frac{s-1}{s}\right)^n\\ &= d_{N+1}^{(m)}\left( \frac{s-1}{s}\right)^{N}+ \frac{1}{(s-1)}\sum_{n=1}^{N}d_n^{(m)}\left( \frac{s-1}{s}\right)^n; 
\end{align*}
where
$$ d_n^{(m)} = \frac{n}{m!}\int_1^{+\infty} \frac{\{x\}\log^{m-1}x\mathcal{L}_n^{(m-1)}(x)}{x^2} \mathrm{d}x = (-1)^{n-1}\binom{n+m-1}{m}\ell_{n}^{(m-1)}.$$
The fact that
$$d_N^{(m-1)} = o\left(N^{\frac{m+1}{2}} \right), \quad \text{as} \ N \to +\infty $$
is sufficient to justify that
 $$ \sum_{n=0}^{+\infty}\Delta_n^{(m)}\left( \frac{1-s}{s}\right)^n =  \frac{1}{s-1}\sum_{n=1}^{+\infty}d_n^{(m)}\left( \frac{s-1}{s}\right)^n = -  \frac{1}{s-1}\sum_{n=1}^{+\infty} \binom{n+m-1}{m}\ell_{n}^{(m-1)} \left( \frac{1-s}{s}\right)^n.$$
Thus,
$$\frac{1}{m!}\zeta^{(m)}(s) =  \frac{(-1)^m}{\left(s-1\right)^{m+1}} + \frac{(-1)^m}{s^{m+1}} \sum_{n=1}^{+\infty}\binom{n+m-1}{m}\ell_{n}^{(m-1)}\left( \frac{1-s}{s}\right)^{n-1}. $$
\subsection{Proof of Theorem \ref{THELA}. }
For any given complex number $s$ in the half-plane $\Re s >0,$ the function $x\mapsto \log x/x^{s-\frac{1}{2}}$ belongs to the Hilbert space $\mathcal{H}_0.$ Moreover, we have, by using \eqref{GF0}, for any given $n\in \mathbb{N}_0$ and $\Re s > 0$
$$
\int_1^{+\infty} \frac{\log x \mathcal{L}_n(x)}{x^{s-\frac{1}{2}}}\frac{\mathrm{d}x}{x^2} = \frac{1}{s+\frac{1}{2}}\sum_{k=0}^{+\infty}\int_1^{+\infty}\frac{\log x\mathcal{L}_n(x)\mathcal{L}_k(x)}{x^2} \mathrm{d}x \left( \frac{s-\frac{1}{2}}{s+\frac{1}{2}}\right)^k .
$$
For $n=0,$ we have for any value of $k\in \mathbb{N}_0,$
$$ \int_1^{+\infty}\frac{\log x\mathcal{L}_0(x)\mathcal{L}_k(x)}{x^2} \mathrm{d}x = \int_1^{+\infty}\frac{\log x\mathcal{L}_k(x)}{x^2} \mathrm{d}x = \delta_{k,0} - \delta_{k,1}; 
$$
and for $n \in \mathbb{N},$
\begin{align*}
\int_1^{+\infty}\frac{\log x\mathcal{L}_n(x)\mathcal{L}_k(x)}{x^2} \mathrm{d}x &= \int_1^{+\infty}\frac{\log x(\mathcal{L}_n^{(1)}(x)- \mathcal{L}_{n-1}^{(1)}(x))\mathcal{L}_k(x)}{x^2} \mathrm{d}x \\ &= (n+1)(\delta_{k,n}-\delta_{k,n+1}) - n(\delta_{k,n-1} - \delta_{k,n}). 
\end{align*}
Remark that the case for $n=0$ may be included in the general case. Thus,
$$\int_1^{+\infty} \frac{\log x \mathcal{L}_n(x)}{x^{s-\frac{1}{2}}}\frac{\mathrm{d}x}{x^2} = \frac{1}{\left(s+\frac{1}{2}\right)^2} \left( (n+1)\left( \frac{s-\frac{1}{2}}{s+\frac{1}{2}}\right)^n - n\left( \frac{s-\frac{1}{2}}{s+\frac{1}{2}}\right)^{n-1}\right) .$$
Now, since the function $\varphi(x) = \sqrt{x}\{x\}/\log x$ belongs to the Hilbert space $\mathcal{H}_0,$ then the following expansion
$$\varphi(x) = \sum_{n=0}^{+\infty} c_n \mathcal{L}_n(x) $$
holds in $\mathcal{H}_0,$ where 
$$ c_n = \int_1^{+\infty}\frac{\varphi(x)\mathcal{L}_n(x)}{x^2}\mathrm{d}x = \int_1^{+\infty}\frac{\{x\}\mathcal{L}_n(x)}{\log x} \frac{\mathrm{d}x}{x^{\frac{3}{2}}} .$$
Hence,
\begin{align*}
\zeta(s) - \frac{s}{s-1} &= -s \int_1^{+\infty} \frac{\{x\}}{x^{s+1}}\mathrm{d}x = -s \int_1^{+ \infty} \varphi(x)\frac{\log x}{x^{s-\frac12}}\frac{\mathrm{d}x}{x^2} \\ &=-s \sum_{n=0}^{+ \infty} c_n \int_1^{+\infty} \frac{\log x \mathcal{L}_n(x)}{x^{s-\frac{1}{2}}}\frac{\mathrm{d}x}{x^2} \\ &=  -\frac{s}{\left(s+\frac{1}{2}\right)^2}\sum_{n=0}^{+\infty} c_n\left( (n+1)\left( \frac{s-\frac{1}{2}}{s+\frac{1}{2}}\right)^n - n\left( \frac{s-\frac{1}{2}}{s+\frac{1}{2}}\right)^{n-1}\right) \\ &= -\frac{s}{\left(s+\frac{1}{2}\right)^2} \sum_{n=0}^{+\infty}(n+1)(c_n - c_{n+1})\left( \frac{s-\frac{1}{2}}{s+\frac{1}{2}}\right)^n. 
\end{align*}
Since, for any $n \in \mathbb{N}_0$
\begin{align*}
(n+1)(c_n - c_{n+1}) &= \int_1^{+\infty} \frac{\{x\}\mathcal{L}_n^{(1)}(x)}{x^{\frac32}} \mathrm{d}x \\ &= \frac12 \int_1^{+\infty} \frac{\mathcal{L}_n^{(1)}(x)}{x^{\frac32}} \mathrm{d}x + \int_1^{+\infty}\frac{\left(\{x\}-\frac12\right)\mathcal{L}_n^{(1)}(x)}{x^{\frac32}} \mathrm{d}x \\ &= \frac{1+(-1)^n}{2} \quad + \quad \sum_{k=0}^n \alpha_k,
\end{align*}
with
$$ \alpha_k = \int_1^{+\infty} \frac{\left(\{x\}-\frac{1}{2}\right)\mathcal{L}_k(x)}{x^{\frac32}} \mathrm{d}x;$$
then, for any complex number $s$ such that $\Re s>0,$
\begin{align*}
\zeta(s) - \frac{s}{s-1} &= -\frac{s}{\left(s+\frac{1}{2}\right)^2} \sum_{n=0}^{+\infty} \frac{1+(-1)^n}{2}\left( \frac{s-\frac{1}{2}}{s+\frac{1}{2}}\right)^n - \frac{s}{\left(s+\frac{1}{2}\right)^2} \sum_{n=0}^{+\infty} \left(\sum_{k=0}^n\alpha_k\right)\left( \frac{s-\frac{1}{2}}{s+\frac{1}{2}}\right)^n \\ &= - \frac{1}{2} - \frac{s}{\left(s+\frac{1}{2}\right)^2} \sum_{n=0}^{+\infty} \left(\sum_{k=0}^n\alpha_k\right)\left( \frac{s-\frac{1}{2}}{s+\frac{1}{2}}\right)^n.
\end{align*}
Finally, the fact that
$$\sum_{n=0}^{N} \left(\sum_{k=0}^n\alpha_k\right)\left( \frac{s-\frac{1}{2}}{s+\frac{1}{2}}\right)^n = - \left( s - \frac12 \right)\left( \frac{s-\frac{1}{2}}{s+\frac{1}{2}}\right)^N\left(\sum_{k=0}^N\alpha_k\right) + \left(s+\frac12 \right) \sum_{n=0}^N \alpha_n \left( \frac{s-\frac{1}{2}}{s+\frac{1}{2}}\right)^n, $$
for any given $N \in \mathbb{N};$ and by remarking that 
\begin{equation}
\sum_{k=0}^N\alpha_k = \mathrm{o}(N+2), \qquad \mbox{as} \quad N \to +\infty 
\label{Opn}
\end{equation}
allows us to write, for all $\Re s >0$
$$\zeta(s) = \frac{s}{s-1} - \frac{1}{2} - \frac{s}{s+\frac{1}{2}} \sum_{n=0}^{+\infty} \alpha_n \left( \frac{s-\frac{1}{2}}{s+\frac{1}{2}}\right)^n.$$
Notice that, one can justify the estimation in \eqref{Opn} by using integration by parts; namely, for a sufficiently large $n,$ we have
\begin{align*}\sum_{k=0}^n \alpha_k &= \int_1^{+\infty}\frac{\left(\{x\}-\frac12\right)\mathcal{L}_n^{(1)}(x)}{x^{\frac32}} \mathrm{d}x \\ &= \frac{1}{2}\int_1^{+\infty}\frac{\left(\{x\}^2-\{x\}\right)\mathcal{L}_n^{(2)}(x)}{x^{\frac52}} \mathrm{d}x + \frac14 \int_1^{+\infty}\frac{\left(\{x\}^2-\{x\}\right)\mathcal{L}_n^{(1)}(x)}{x^{\frac52}} \mathrm{d}x \\ &= o\left(\sqrt{(n+2)(n+1)}\right) + o\left( \sqrt{n+1}\right) = o(n+2).
 \end{align*}
Actually, it is not easy to study the growth rate of the coefficients $\alpha_n;$ however, one can remark that
$$
(-1)^{n-1} \alpha_n = \frac{1}{2\pi i} \int_{0 < \Re s < \frac12}\left( \frac{2\zeta\left( s - \frac12 \right)+ 1}{(2s-1)(1-s)}- \frac{2}{(2s-3)(1-s)}\right)\left(\frac{s}{1-s} \right)^n\mathrm{d}s .
$$
Since, for any complex number $s=\sigma+it$ such that $\Re s = \sigma \in (\varepsilon ,1/2)$ (with a small $\varepsilon>0$) we have $|s/(1-s)|< 1$ and 
$$ \frac{2\zeta\left( s - \frac12 \right)+ 1}{(2s-1)(1-s)}- \frac{2}{(2s-3)(1-s)} = O\left( \frac{1}{|t|^{1+\varepsilon}}\right), \qquad  \mbox{as} \quad |t|\to +\infty.$$
Thus, the absolute convergence of the integral representation of $(\alpha_n)$ allows as to state that
$$\lim_{n \to + \infty} \alpha_n = 0; $$
and also that 
$$\sum_{k=0}^{+\infty}(-1)^{k}\alpha_k = \frac{1}{2\pi i} \int_{\varepsilon < \Re s < \frac12}\left( \frac{2\zeta\left( s - \frac12 \right)+ 1}{(2s-1)^2}- \frac{2}{(2s-3)(2s-1)}\right)\mathrm{d}s < + \infty . $$

In general, the study of coefficients $(\alpha_n)$ is very interesting in order to better understand the value-distribution of the Riemann zeta function. Moreover, we point out that one can easly find its explicit expression in function of derivatives of $\zeta(s)$ at $s=1/2;$ i.e. $\zeta^{(k)}(1/2)$ ($k\in \mathbb{N}_0$). However, to find the upper bound of $(\alpha_n)$ still not obvious.

\section{The explicit values of the $\mathcal{H}_m-$norm of the fractional part function}

We start our concluding section by improving some identities involving the coefficients $(\ell_0^{(m)})_{m \geq 0}.$ The first explicit formula in terms of Stieltjes constants $(\gamma_k)_{k \geq 0}$ has been proven to be \eqref{L0M}. Moreover, the sequences $(\ell_n^{(m)})_{n \geq 0}$ and $(\ell_0^{(n+m)})_{n \geq 0}$ are related, for any given $m \in \mathbb{N}_0,$ by the following invertible formula 
$$\ell_n^{(m)} = \sum_{k=0}^n \binom{n}{k}(-1)^{n-k}\ell_0^{(m+k)} := \mathfrak{B}(\ell_0^{(n+m)});$$
where $\mathfrak{B}$ denotes the well-known Binomial transform and its inverse is 
$$\ell_0^{(m+n)} = \mathfrak{B}^{-1}(\ell_n^{(m)}) = \sum_{k=0}^{n} \binom{n}{k}\ell_k^{(m)} .$$
More generally, if we consider the exponential generating functions $F_m$ and $G_m,$ respectively,  for the coefficients $(\ell_n^{(m)})_{n \geq 0},$ ($m \in \mathbb{N}_0$); namely,
 $$F_m(z):= \sum_{n=0}^{+\infty}\frac{\ell_n^{(m)}}{n!}z^n \quad \mbox{and} \quad G_m(z):= \sum_{n=0}^{+\infty}\frac{\ell_0^{(m+n)}}{n!}z^n,$$
for some complex number $z;$ then we have     
\begin{align*}F_m(z) =\sum_{n = 0}^{+\infty} \frac{\ell_n^{(m)}}{n!}z^n &= \sum_{n=0}^{+\infty}\sum_{k=0}^{n}\frac{(-z)^{n-k}}{(n-k)!}\frac{z^{k}\ell_0^{(m+k)}}{k!} \\ &= \left( \sum_{n=0}^{+\infty}\frac{(-z)^n}{n!}\right) \left( \sum_{n=0}^{+\infty}\frac{\ell_0^{(m+n)}z^n}{n!}\right)\\ &= G_m(z)e^{-z}.
\end{align*}
It seems a good piece of work to study the functions $F_m$ and $G_m$ as the exponential generating function of the coefficients $(\ell_n^{(m)})_{n \in \mathbb{N}}$ and $(\ell_0^{(n)}).$ We can remark by a first look, that the function $(F_m)$ behaves as an hypergeometric function similar to the well-known Bessel function $(J_m).$ More precisely, one can think exactly on the Hankel transforms of the fractional part function.

Now, we conclude with the explicit values of integrals involving the fractional part function. We showed in the previous section that the  $\mathcal{H}_m-$norm ($m\in \mathbb{N}_0$) of the fractional part function is
$$\|\{\cdot\}\|_{(m)}^2 = \frac{1}{m!}\int_1^{+\infty}\frac{\{x\}^2\log^m x}{x^2}\mathrm{d}x = \sum_{n=0}^{+\infty}\binom{n+m}{m}\left|\ell_n^{(m)}\right|^2 .$$
 Hence, its explicit value is,
\begin{theorem}
For any given $m\in \mathbb{N}_0,$ we have
$$\|\{\cdot\}\|_{(m)}^2 = 1 - (-1)^m + 2\sum_{k=0}^{m+1}\frac{(-1)^k\zeta^{(k)}(0)}{k!} - \sum_{k=0}^m\frac{\gamma_k}{k!}.$$
\end{theorem}
\begin{proof}
By using the integration by parts on \eqref{ZFM}, we have for any $\Re s > 0,$
\begin{align}
    \zeta(s) - \frac{s}{s-1} &= -\frac{1}{2} - s\int_1^{+\infty} \frac{\{x\} - \frac{1}{2}}{x^{s+1}}\mathrm{d}x \nonumber\\ &= - \frac{1}{2} - \frac{s(s+1)}{2} \int_1^{+\infty}\frac{\{x\}^2-\{x\}}{x^{s+2}}\mathrm{d}x.
\label{ZFG} \end{align}
 One can obtain the generalization of the formula \eqref{ZFG} by using the well-known Euler-Maclaurin summation, see for example \cite[Ch. 6]{Edw}. Notice that, the integral in \eqref{ZFG} is absolutely convergent for any complex number $s$ in the half-plane $\Re s > -1,$ and defines an analytic continuation of the Riemann zeta function to the half-plane $\Re s> -1.$ Thus by a similar reasoning as in the proof of Theorem \ref{ThEL}, we have
 $$(-1)^{m-1}\int_1^{+ \infty} \frac{\left(\{x\}^2-\{x\}\right)\log^m x}{x^{2}}\mathrm{d}x = \lim_{s \to 0} \frac{\mathrm{d}^m}{\mathrm{d}s^m}\left(2\frac{\zeta(s)+\frac{1}{2}}{s(s+1)} - \frac{2}{s^2-1}\right) .$$
 Since $\zeta(s)$ is analytic near $s=0,$ then for very small values of $|s|$ we have
 $$\zeta(s) = -\frac{1}{2} + \sum_{k=1}^{+\infty} \frac{\zeta^{(k)}(0)}{k!}s^k  .$$
 Hence,
 \begin{align*}  \lim_{s \to 0} \frac{\mathrm{d}^m}{\mathrm{d}s^m}\left(2\frac{\zeta(s)+\frac{1}{2}}{s(s+1)}\right) &= \lim_{s \to 0} \frac{\mathrm{d}^m}{\mathrm{d}s^m}\left(2\sum_{k=0}^{+\infty}\frac{\zeta^{(k+1)}(0)}{(k+1)!}\frac{s^k}{s+1} \right) \\ &= 2(-1)^mm! \sum_{k=0}^{m} \frac{(-1)^k}{(k+1)!}\zeta^{(k+1)}(0).
 \end{align*}
 Consequently,
 $$ \frac{1}{m!}\int_1^{+\infty} \frac{\left(\{x\}^2-\{x\}\right)\log^m x}{x^{2}}\mathrm{d}x =  (-1)^{m-1} + 2 \sum_{k=0}^{m+1} \frac{(-1)^k}{k!}\zeta^{(k)}(0) ;$$
 with $\zeta(0) = -1/2$ and $\zeta'(0)= -\log(2\pi)/2.$ Thus, the fact that
$$ \ell_0^{(m)} = -\frac{1}{m!}\int_1^{+\infty} \frac{\{x\}\log^mx}{x^2}\mathrm{d}x = -1 + \sum_{k=0}^m \frac{\gamma_k}{k!},$$
for any given $m \in \mathbb{N}_0,$ completes the proof.
\end{proof}

Finally, we would like to notice that one can extend our results to a large class of arithmetic functions those belonging in the Hilbert space $\mathcal{H}_m.$ Hence, by a similar reasoning as in the proof of Theorem \ref{ThEL} one can obtain very interesting and new expansions of its associated Dirichlet series. In fact, the study of the Fourier-Laguerre coefficients of a given arithmetic function in $\mathcal{H}_m$ may provides many answers of many problems on the values-distribution of associated Dirichlet series. See the discussion on Lindel\"of hypothesis in \cite[\S 4.2]{Ela}.


\end{document}